\documentclass[a4paper]{amsart}
\usepackage[utf8]{inputenc}
\usepackage{hyperref}
\usepackage{microtype}

\newtheorem{thm}{Theorem}[section]
\newtheorem{prop}[thm]{Proposition}
\newtheorem{lemma}[thm]{Lemma}
\newtheorem{cor}[thm]{Corollary}

\theoremstyle{definition}

\newtheorem{remark}[thm]{Remark}
\newtheorem{exmpl}[thm]{Example}

\newcommand{\N}{\mathbb N}

\newcommand{\ito}{\hookrightarrow}
\newcommand{\q}{\mathfrak q}
\newcommand{\p}{\mathfrak p}
\newcommand{\m}{\mathfrak m}

\DeclareMathOperator{\Ext}{Ext}
\DeclareMathOperator{\Tor}{Tor}

\DeclareMathOperator{\Spec}{Spec}

\let\mc\mathcal

\begin{document}
\title{Countably generated flat modules are quite flat}

\author{Michal Hrbek}
\address{Institute of Mathematics of the Czech Academy of Sciences, \v Zitn\'a 25, 115\,67 Prague~1, Czech Republic}
\email{hrbek@math.cas.cz}
\thanks{The first author's research is supported by research plan RVO:~67985840.}

\author{Leonid Positselski}
\address{Institute of Mathematics of the Czech Academy of Sciences, \v Zitn\'a 25, 115\,67 Prague~1, Czech Republic}
\email{positselski@yandex.ru}
\thanks{The second author's research is supported by research plan RVO:~67985840.}

\author{Alexander Sl\'avik}
\address{Department of Algebra, Charles University, Faculty of Mathematics and Physics, Sokolovsk\'a 83, 186\,75 Prague 8, Czech Republic}
\email{slavik.alexander@seznam.cz}
\thanks{The third author's research is supported from the grant GA \v CR 17-23112S of the Czech Science Foundation, from the grant SVV-2017-260456 of the SVV project and from the grant UNCE/SCI/022 of the Charles University Research Centre.}

\begin{abstract}
We prove that if $R$ is a commutative Noetherian ring, then every countably generated flat $R$-module is quite flat, i.e., a direct summand of a transfinite extension of localizations of $R$ in countable multiplicative subsets. We also show that if the spectrum of $R$ is of cardinality less than $\kappa$, where $\kappa$ is an uncountable regular cardinal, then every flat $R$-module is a transfinite extension of flat modules with less than $\kappa$ generators. This provides an alternative proof of the fact that over a commutative Noetherian ring with countable spectrum, all flat modules are quite flat.
More generally, we say that a commutative ring is CFQ if every countably presented flat $R$-module is quite flat. We show that all von~Neumann regular rings and all $S$\nobreakdash-almost perfect rings are CFQ. A zero-dimensional local ring is CFQ if and only if it is perfect. A domain is CFQ if and only if all its proper quotient rings are CFQ. A valuation domain is CFQ if and only if it is strongly discrete.
\end{abstract}

\maketitle

\section{Introduction}

Over any ring, the Govorov--Lazard Theorem provides a description of flat modules
as direct limits of finitely generated free modules. However, this description, while sometimes useful, does not give much insight into the properties of flat modules; for example, for the ring of integers, the theorem says that every torsion-free abelian group is the direct limit of finitely generated free abelian groups, which is clear from the fact that finitely generated subgroups of torsion-free groups are free.
However, a more informative description of torsion-free groups is available, going back to Trlifaj \cite{Trl} with a generalization due to Bazzoni--Salce \cite{BS} (see the beginning of the introduction to \cite{PS}). So one wishes, and sometimes can have,
a more precise description of flat modules.

The descriptions of classes of modules (in particular, flat modules) that we have in mind are formulated in terms of transfinite extensions. Recall that if $\mc C$ is a class of $R$-modules, then an $R$-module $M$ is a \emph{transfinite extension} of modules from $\mc C$ if there is a well-ordered chain of submodules of $M$, $(M_\alpha \mid \alpha \leq \sigma)$, such that $M_0 = 0$, $M_\sigma = M$, $M_\alpha = \bigcup_{\beta<\alpha} M_\beta$ for every limit ordinal $\alpha\le\sigma$, and the quotient module $M_{\alpha+1}/M_{\alpha}$ is isomorphic to an element of $\mc C$ for every for every $\alpha < \sigma$. We also say that $M$ is \emph{$\mc C$-filtered} in that case.

In particular, the class of quite flat modules over a commutative ring $R$ was defined in the paper \cite{PS} as follows. We say that an $R$-module $C$ is \emph{almost cotorsion} if $\Ext_R^1(S^{-1}R, C) = 0$ for all (at most) countable multiplicative subsets $S \subseteq R$. An $R$\nobreakdash-module $F$ is said to be \emph{quite flat} if $\Ext_R^1(F, C) = 0$ for all almost cotorsion $R$\nobreakdash -modules $C$. By \cite[Corollary 6.14]{GT}, this means that quite flat modules are precisely the direct summands of transfinite extensions of modules of the form $S^{-1}R$, where $S$ is a countable multiplicative subset of $R$.

It was shown in \cite{PS} that all flat modules over a commutative Noetherian ring with a countable spectrum are quite flat.  In this paper we prove the following generalization of this result: For any commutative Noetherian ring, any countably generated flat module is quite flat. Then we offer an alternative proof of the mentioned theorem from \cite{PS}, by explaining how to deduce the description of arbitrary flat modules over a commutative Noetherian ring with countable spectrum from the description of countably generated flat modules.

To be more specific, the theorem that all countably generated flat modules over a commutative Noetherian ring are quite flat is proved in Section \ref{section-countable-is-quite-flat}.  In Section \ref{section-bounded-spectrum} we work more generally with a commutative Noetherian ring $R$ whose spectrum has cardinality smaller than $\kappa$, where $\kappa$ is a regular uncountable cardinal.  In this setting, we prove that every flat $R$-module is a transfinite extension of $<\kappa$-generated flat $R$-modules.

In Section \ref{section-cfq} we discuss (non-Noetherian) commutative rings $R$ over which all countably presented flat modules are quite flat.
We call such rings $R$ \emph{CFQ rings}.
In particular, all von Neumann regular commutative rings and all $S$\nobreakdash-almost perfect commutative rings in the sense of the paper \cite{BP} are CFQ.
A zero-dimensional local ring is CFQ if and only if it is perfect, and a one-dimensional local domain is CFQ if and only if it is almost perfect.
A domain is CFQ if and only if all its quotient rings by nonzero ideals are CFQ.
A one-dimensional CFQ domain is always locally almost perfect, but it does not need to be almost perfect.

In Section \ref{section-valuation-domains} we discuss the case of valuation domains, and prove that a valuation domain is CFQ if and only if it is strongly discrete.
In the final Section \ref{section-finitely-quite-flat}, we show that over locally perfect commutative rings all finitely generated, countably presented flat modules are quite flat.

We are grateful to Jan Trlifaj for the suggestion to include Remarks \ref{osofsky-vs-Q-remark} and \ref{countable-vN-regular}.
We also want to thank the anonymous referee for careful reading of the manuscript and several helpful suggestions on the improvement of the exposition.

\section{Noetherian rings} \label{section-countable-is-quite-flat}

In this section we prove the main result promised in the title of the paper: \emph{All countably generated flat modules over a Noetherian commutative ring are quite flat.} There are two main ingredients: Firstly, there is the ``Main Lemma'' from \cite{PS}, which makes it possible to check whether a module is quite flat by reducing the question to rings of smaller Krull dimension. We recall the statement for the convenience of the reader.

\begin{lemma}[{\cite[Main Lemma 1.18]{PS}}]\label{main-lemma}
Let $R$ be a Noetherian commutative ring and $S \subseteq R$ be a countable multiplicative subset. Then a flat $R$-module $F$ is quite flat if and only if the $R/sR$-module $F/sF$ is quite flat for all $s \in S$ and the $S^{-1}R$-module $S^{-1}F$ is quite flat.
\end{lemma}

The second ingredient is a lemma ensuring that there is always a suitable countable multiplicative subset to be used in Lemma \ref{main-lemma}.
Before formulating the lemma, we prove a proposition, which holds even for non-Noetherian commutative rings.

\begin{prop} \label{countable-subset}
Let $R$ be a commutative ring and $F$ a countably presented flat $R$-module.
Let $T\subseteq R$ be a multiplicative subset such that $T^{-1}F$ is a projective $T^{-1}R$-module.
Then there is a countable multiplicative subset $S\subseteq T$ such that $S^{-1}F$ is a projective $S^{-1}R$-module.
\end{prop}

\begin{proof}
It is a standard fact that countably presented flat modules have projective dimension at most one. Furthermore, by \cite[Corollary 2.23]{GT}, $F$ is the cokernel of a monomorphism between countable-rank free $R$-modules; let $f\colon R^{(\N)} \to R^{(\N)}$ be this monomorphism.
The monomorphism $T^{-1}f\colon T^{-1}R^{(\N)} \to T^{-1}R^{(\N)}$ splits by assumption; let $g \colon T^{-1}R^{(\N)} \to T^{-1}R^{(\N)}$ be a map of $T^{-1}R$-modules such that $g(T^{-1}f) = \mathrm{id}_{T^{-1}R^{(\N)}}$.

The maps $T^{-1}f$ and $g$, being maps between free modules, can be represented by column-finite matrices of countable size of elements of $T^{-1}R$ (provided we view the elements of free modules as column vectors); denote by $A$ and $B$ the corresponding matrices, respectively, and let $E$ be the identity matrix of countable size. Then $BA - E = 0$, a matrix equation which translates into countably many equations in $T^{-1}R$. Every such equation becomes a valid equation in $R$ after multiplying by an appropriate element of $T$; pick such an element for each of the equations and let $V \subseteq T$ be the set of all these elements. Further, let $D \subseteq T$ be the set of all denominators appearing in the entries of the matrix $B$.

Both $V$ and $D$ are countable sets, therefore the multiplicative subset $S \subseteq R$ generated by $V \cup D$ is countable, too. As $D \subseteq S$, the entries of $B$ are naturally elements of $S^{-1}R$ and since $V \subseteq S$, the matrix equation $BA - E = 0$ holds in $S^{-1}R$, too. Hence $B$ defines a splitting of the monomorphism $S^{-1}f\colon S^{-1}R^{(\N)} \to S^{-1}R^{(\N)}$, the cokernel of which is $S^{-1}F$, which is therefore a projective $S^{-1}R$-module.
It remains to observe that $V \cup D \subseteq T$ implies $S \subseteq T$.
\end{proof}

\begin{lemma}\label{mult-set-exists}
Let $R$ be a Noetherian commutative ring and $F$ a countably generated flat module. Then there is a countable multiplicative subset $S \subseteq R$ such that $S \cap \q = \emptyset$ for every minimal prime ideal\/ $\q$ of $R$ and $S^{-1}F$ is a projective $S^{-1}R$-module.
\end{lemma}

\begin{proof}
Let $\q_1, \dots, \q_k$ be the minimal prime ideals of $R$ and put $T = R \setminus (\q_1 \cup \dots \cup \q_k)$. Then $T$ is a multiplicative subset intersecting all but the minimal primes of $R$, hence $T^{-1}R$ is an Artinian ring. It follows that $T^{-1}F$ is a projective $T^{-1}R$-module.

Since $R$ is Noetherian, every countably generated module is countably presented, so, by Proposition~\ref{countable-subset}, there is a countable multiplicative subset $S\subseteq T$ such that $S^{-1}F$ is a projective $S^{-1}R$-module.  Finally, the inclusion $S\subseteq T$ implies $S \cap \q = \emptyset$ for every minimal prime $\q$ by the choice of $T$.
\end{proof}

We are now ready to prove the main result.

\begin{thm}\label{main-theorem}
Let $R$ be a Noetherian commutative ring and $F$ a countably generated flat module. Then $F$ is quite flat.
\end{thm}
\begin{proof}
The strategy, ``Noetherian induction'', is borrowed from the proof of \cite[Theorem 1.17]{PS}. Assume that $F_0 = F$ is a countably generated flat module which is \emph{not} quite flat. By Lemma \ref{mult-set-exists}, there is a countable multiplicative subset $S_0$ not intersecting the minimal primes of $R_0 = R$ and such that $S_0^{-1}F_0$ is a projective $S_0^{-1}R_0$-module. Therefore, by Lemma \ref{main-lemma}, since $F_0$ is not quite flat, there is $s_0 \in S_0$ such that $F_0/s_0F_0$, which is a countably generated flat $R_0/s_0R_0$-module, is not a quite flat $R_0/s_0R_0$-module.

The ring $R_1 = R_0/s_0R_0$ is a Noetherian commutative ring and by Lemma \ref{mult-set-exists}, we again obtain a multiplicative subset $S_1 \subseteq R_1$ with analogous properties with respect to the ring $R_1$ and the $R_1$-module $F_1 = F_0/s_0F_0$. Similarly, Lemma \ref{main-lemma} produces an element $s_1 \in S_1$ such that $F_1/s_1F_1$ is not a quite flat $R_1/s_1R_1$-module. Repeating this procedure, we obtain an infinite sequence $s_0 \in R_0$, $s_1 \in R_1$ etc.

Denote by $\tilde s_n \in R$ any  preimage of $s_n \in R_n$ for every $n \in \N_0$ and let $I_n$ be the ideal generated by $\tilde s_0, \dots, \tilde s_n$. Since each $s_n$ is picked from $S_n$, which avoids the minimal primes of $R_n$, the chain of ideals $I_0, I_1, \dots$ is strictly increasing, which contradicts Noetherianity of $R$. We conclude that $F$ is a quite flat $R$-module.
\end{proof}

\begin{cor} \label{countable-ordinals}
Let $R$ be a Noetherian commutative ring. Then an $R$-module $F$ is a countably generated flat module if and only if $F$ is a direct summand of a transfinite extension, indexed by a countable ordinal, of $R$-modules of the form $S^{-1}R$, where $S$ ranges over countable multiplicative subsets of $R$.
\end{cor}
\begin{proof}
The ``if'' part is clear. As for the ``only if'' part, by Theorem \ref{main-theorem}, $F$ is quite flat, so as pointed out in \cite[\S1.6]{PS}, it is a direct summand of a transfinite extension $E$ of $R$-modules of the form $S^{-1}R$, where $S$ are countable multiplicative subsets. Now by the Hill Lemma \cite[Theorem 7.10]{GT} (taking $\kappa = \aleph_1$, $M=E$, $N=0$, and $X$ a countable generating set of $F$ in (H4)), $F$ is in fact contained in a countably generated module $E' \subseteq E$, again filtered by modules of the form $S^{-1}R$. An inspection of the last paragraph of the proof of \cite[Theorem 7.10]{GT} then shows that the ordinal type of the filtration of $E'$ is countable.
\end{proof}

\section{Noetherian rings with bounded cardinality of spectrum}\label{section-bounded-spectrum}

Let $R$ be a Noetherian commutative ring with countable spectrum; then, by \cite[Theorem 1.17]{PS}, all flat $R$-modules are quite flat. In particular, all flat $R$-modules are transfinite extensions of countably generated flat modules. This result can be proved directly, which we are going to do now.

The following lemma is standard and holds also in the non-commutative case once the obvious alterations are made. We spell it out so we can refer to it easily.

\begin{lemma}\label{tensoring-equivalents}
Let $R$ be a commutative ring and $M$, $F$ $R$-modules such that $M \subseteq F$ and $I$ an ideal of $R$. The following are equivalent:
\begin{enumerate}
    \item the map $M \otimes_R (R/I) \to F \otimes_R (R/I)$ is injective,
    \item the map $M/IM \to F/IF$ is injective,
    \item $IF \cap M \subseteq IM$ (in which case necessarily $IF \cap M = IM$).
\end{enumerate}
\end{lemma}
\begin{proof}
(1) $\Leftrightarrow$ (2): By tensoring the short exact sequence $0 \to I \to R \to R/I \to 0$ by an $R$-module $A$ and noting that the image of $A \otimes_R I \to A \otimes_R R \cong A$ is precisely $IA$, we get that $A \otimes_R (R/I)$ is naturally isomorphic to $A/IA$ for any $A$ and $I$.

(2) $\Leftrightarrow$ (3): The kernel of the composition $M \ito F \to F/IF$ is precisely $IF \cap M$, so $M/IM \to F/IF$ is injective if and only if $IF \cap M \subseteq IM$, and since $IF \cap M \supseteq IM$ holds always, this is also equivalent to $IF \cap M = IM$.
\end{proof}

The following is again a known result: The general (not necessarily commutative) case is e.g.\ \cite[Lemma 19.18]{AF}, and the Noetherian case was established in \cite[Lemma 4.2 and the following paragraph]{E}, although the proof is quite different.

\begin{lemma}\label{primes-are-sufficient}
Let $R$ be a commutative ring, $F$ a flat $R$-module and $M$ a submodule of $F$. Then $M$ is a pure submodule of $F$ if and only if for each finitely generated ideal $I$ of $R$, the natural map
\[ M \otimes_R (R/I) \to F \otimes_R (R/I) \]
is injective. If $R$ is a Noetherian commutative ring, then it suffices to take for $I$ the prime ideals of $R$.
\end{lemma}
\begin{proof}
If the inclusion of $M$ into $F$ is pure, then it stays injective after tensoring with any 
$R$-module, in particular with $R/I$.

On the other hand, since $F$ is flat, $M$ is a pure submodule if and only if the factormodule $C = F/M$ is flat, i.e., $\Tor_1^R(C, A) = 0$ for every $R$-module $A$. However, the vanishing of $\Tor$ is preserved by transfinite extensions, and since every $R$-module is a transfinite extension of cyclic modules, it suffices to verify that $\Tor_1^R(C, R/I) = 0$ for every ideal of $R$. Moreover, since every ideal is the directed union of its finitely generated subideals, every cyclic module is the direct limit of modules of the form $R/I$ for $I$ finitely generated, and since $\Tor$ commutes with direct limits, we see that it is enough to test that $\Tor_1^R(C, R/I) = 0$ for every finitely generated ideal of $R$.
Since $F$ is flat, $\Tor_1^R(F, R/I) = 0$, so $\Tor_1^R(C, R/I)$ is precisely the kernel of the map $M \otimes_R (R/I) \to F \otimes_R (R/I)$, hence it is zero if and only if this map is injective.

If $R$ is a Noetherian ring, then every module is a transfinite extension of modules of the form $R/\p$, where $\p$ is a prime ideal of $R$. Therefore it suffices to check only that $\Tor_1^R(C, R/\p) = 0$ and the argument concludes in the same way.
\end{proof}

If $F$ is not flat, Lemma \ref{primes-are-sufficient} (even its weaker form) is no longer valid even in the Noetherian case, which we are most interested in:

\begin{exmpl}\label{example-primes-not-sufficient}
Let $k$ be a field, $k[x, y]$ the ring of polynomials in two variables and $R = k[x, y]/(x^2, xy, y^2)$. We will denote the cosets of $x$ and $y$ in $R$ again by $x$ and $y$ for simplicity. Let $F$ be a $k$-vector space with five-element basis $\{a, b, s, t, e\}$, on which we define the actions of $x$ and $y$ as follows: $xs = ys = 0$, $xt = yt = 0$, $xa = s$, $ya = t$, $xb = t$, $yb = 0$, $xe = s$, $ye = 0$; it is easy to see that this makes $F$ an $R$-module. Furthermore, the $k$-subspace generated by $\{a, b, s, t\}$ is an $R$-submodule of $F$, which we denote by $M$. We claim that $IM = IF \cap M$ for every ideal $I$ of $R$, but $M$ is not pure in $F$.

Firstly, observe that $F/M$ is the simple $R$-module on which $x$ and $y$ act by zero. Since
\begin{align*}
x(\alpha a + \beta b + \varepsilon e) &= (\alpha + \varepsilon) s + \beta t,\\
y(\alpha a + \beta b + \varepsilon e) &= \alpha t
\end{align*}
for $\alpha, \beta, \varepsilon \in k$, the only $k$-linear combination of $a$, $b$, $e$ annihilated by both $x$ and $y$ is the trivial one. Therefore $k$-linear combinations of $s$ and $t$ are the only elements of $F$ killed by both $x$ and $y$. We conclude that there is no section of the $R$-module projection $F \to F/M$, hence $M$ is not a direct summand and consequently, not a pure submodule of $F$.

Secondly, note that whenever $I$ is an ideal of $R$ such that $I \not\subseteq (y)$, then $s \in IM$: Either $I$ contains an element $i$ with a non-zero absolute term, in which case $is = s$, or $I \subseteq (x, y)$. In the latter case, there are $u, v \in k$, $u \neq 0$ such that $ux + vy \in I$; then one can find $\alpha, \beta \in k$ such that $(ux + vy)(\alpha a + \beta b) = s$ by solving a system of two linear equations with regular matrix.

A typical element $q$ of $IF$ is of the form
\[ q = i_1(m_1 + \varepsilon_1 e) + \dots + i_n(m_n + \varepsilon_n e), \]
where $i_1, \dots, i_n \in I$, $m_1, \dots, m_n \in M$ and $\varepsilon_1, \dots, \varepsilon_n \in k$.
The element $r = (i_1 \varepsilon_1 + \dots + i_n \varepsilon_n)e$ is a linear combination of $s$ and $e$; for $q$ to be in $IF \cap M$, $r$ must be a multiple of $s$, therefore $r \in IM$ by the discussion above. Since $i_1 m_1 + \dots + i_n m_n \in IM$, we conclude that $q \in IM$ as desired.

Finally, if an ideal $I$ satisfies $I \subseteq (y)$, then $IM = IF$ and we are done.
\end{exmpl}

Let $\kappa$ be a regular cardinal. We say that a commutative ring $R$ is \emph{$<\kappa$-Noetherian} if every ideal of $I$ is $<\kappa$-generated. Note that by \cite[Lemma 6.31]{GT}, submodules of $<\kappa$-generated modules over a $<\kappa$-Noetherian ring are $<\kappa$-generated; in particular, every $<\kappa$-generated module is $<\kappa$-presented.

\begin{lemma}\label{single-ideal-purification}
Let $\kappa$ be an uncountable regular cardinal, $R$ a $<\kappa$-Noetherian commutative ring, $I$ an ideal of $R$, $F$ an $R$-module and $X$ a subset of $F$ of cardinality $<\kappa$. Then there is a $<\kappa$-generated submodule $M \subseteq F$ such that $X \subseteq M$ and $IM = IF \cap M$.
\end{lemma}
\begin{proof}
Let $X_0 = X$. Denote by $M_0$ the submodule of $F$ generated by $X_0$; this is a $<\kappa$-generated module. 
Since $R$ is $<\kappa$-Noetherian, the submodule $I F \cap M_0$ of $M_0$ is $<\kappa$-generated, too; let $Y_0$ be a set of cardinality $<\kappa$ generating this module.
Every $y \in Y_0$ can be written as
\[ y = p_1 y_1 + \dots + p_n y_n, \]
where $p_i \in I$ and $y_i \in F$ for $i = 1, \dots, n$. Gathering these $y_i$'s for all $y \in Y_0$, we obtain a subset $Z_0 \subseteq F$ of cardinality $<\kappa$. By the construction, the submodule $M_1 \subseteq F$ generated by $X_0 \cup Z_0$ has the property
$I F \cap M_0 \subseteq I M_1$.

Now repeat this procedure, starting with the set $X_1 = X_0 \cup Z_0$ of cardinality $<\kappa$, obtaining a subset $Z_1 \subseteq F$ of cardinality $<\kappa$. Continuing in this fashion, i.e., repeating the procedure with $X_{i+1} = X_i \cup Z_i$, we obtain an $\N_0$-indexed chain $X_0 \subseteq X_1 \subseteq X_2 \subseteq \ldots$ of subsets of $F$ of cardinality $< \kappa$; let $X$ be its union. Note that the cardinality of $X$ is less than $\kappa$, since $\kappa$ is uncountable and regular. We claim that the submodule $M \subseteq F$ generated by $X$ has the desired property: This is because $M = \bigcup_{n \in \mathbb N_0} M_n$ and
\[ I F \cap M = \bigcup_{n \in \N_0} (I F \cap M_n) \subseteq \bigcup_{n \in \mathbb N_0} I M_{n+1} = I M. \]
\end{proof}

\begin{lemma}\label{kappa-presented-exists}
Let $R$ be a Noetherian commutative ring with spectrum of cardinality less than $\kappa$, where $\kappa$ is an uncountable regular cardinal. Let $F$ be a flat $R$-module and $X$ a subset of $F$ of cardinality $<\kappa$. Then there is a pure submodule $M \subseteq F$ such that $X \subseteq M$ and $M$ is $<\kappa$-generated.
\end{lemma}
\begin{proof}
We prove the lemma by ``iterating Lemma \ref{single-ideal-purification} sufficiently many times'' for each prime ideal of $R$. More precisely, let $\lambda_0$ be the cardinality of the spectrum of $R$.  Put $\lambda=\lambda_0$ if $\lambda_0$ is infinite, and let $\lambda$ be the countable cardinality if $\lambda_0$ is finite. Let $\psi\colon \lambda \to \lambda$ be a surjective function such that for each ordinal $\alpha < \lambda$, the preimage $\psi^{-1}(\alpha)$ is unbounded in $\lambda$. Also let $\{\p_\alpha \mid \alpha < \lambda\}$ be a numbering of the spectrum of $R$ in which every prime ideal of $R$ appears at least once. Finally, put $M_0 = 0$.

Now starting with $X_0 = X$, apply Lemma \ref{single-ideal-purification} with $I = \p_{\psi(0)}$ to get $<\kappa$-generated submodule $M_1 \subseteq F$ such that $X_0 \subseteq M_1$ and
$\p_{\psi(0)} M_1 = \p_{\psi(0)}F \cap M_1$.
More generally, for every $\alpha < \lambda$, if $M_\alpha$ is constructed, let $M_{\alpha+1}$ be the result of applying Lemma \ref{single-ideal-purification} with the prime ideal $\p_{\psi(\alpha)}$ and with a generating set of $M_\alpha$ of cardinality $<\kappa$. For every limit ordinal $\alpha < \lambda$, let $M_\alpha = \bigcup_{\beta<\alpha}M_\beta$; since $\kappa$ is regular, this keeps $M_\alpha$ $<\kappa$-generated for each $\alpha < \lambda$.

Put $M = \bigcup_{\beta<\lambda}M_\beta$. Since $\lambda < \kappa$, $M$ is $<\kappa$-generated. Moreover, by the choice of $\psi$, for every $\gamma < \lambda$, $M$ is the union of those $M_{\alpha+1}$ for which $\psi(\alpha) = \gamma$. Therefore, for every $\gamma < \lambda$,
\[ \p_\gamma F \cap M = \bigcup_{\substack{\alpha < \lambda \\ \psi(\alpha) = \gamma}} (\p_\gamma F \cap M_{\alpha+1}) = \bigcup_{\substack{\alpha < \lambda \\ \psi(\alpha) = \gamma}} \p_\gamma M_{\alpha+1} = \p_\gamma M. \]
We conclude that $\p M = \p F \cap M$ holds for every prime $\p$ as desired, which by Lemmas \ref{tensoring-equivalents} and \ref{primes-are-sufficient} means that $M$ is a pure submodule of $F$.
\end{proof}

Note that in the case $\kappa = \aleph_1$, the lemma can be proved using already known results: Knowing that all flat modules are quite flat in this case \cite[Theorem~1.17]{PS}, it follows easily from the Hill Lemma \cite[Theorem 7.10]{GT}.

\begin{remark}
Let us comment here on the overall situation concerning ``purifications'': It is a standard fact that for a ring $R$ of cardinality not exceeding an infinite cardinal $\lambda$, every $R$-module $F$ and subset $X \subseteq F$ of cardinality at most $\lambda$, there is a pure submodule $M \subseteq F$ of cardinality at most $\lambda$ containing $X$; see e.g.\ \cite[Lemma 2.25(a)]{GT}. Lemma \ref{kappa-presented-exists} shows that when $R$ is commutative Noetherian and $F$ is flat, then instead of the cardinality of the ring, one can take a potentially sharper bound, the cardinality of the spectrum (which, for Noetherian rings, cannot exceed the cardinality of the ring). This is thanks to Lemma \ref{primes-are-sufficient}.

Example \ref{example-primes-not-sufficient} shows that when enlarging arbitrary submodules of non-flat modules to pure submodules, one has to add more than just ``divisors'', in particular, one cannot rely on Lemma \ref{primes-are-sufficient}. However, we do not know whether Lemma \ref{kappa-presented-exists} holds for non-flat modules over commutative Noetherian rings or not.
\end{remark}

\begin{remark} \label{osofsky-vs-Q-remark}
In the special case when $F$ is a flat and  Mittag-Leffler module (see e.g.\ \cite{EGPT} or \cite{GT} for the definition), a stronger result than Lemma~\ref{kappa-presented-exists} is known \cite[Lemma 2.7(2)]{EGPT}:
For any ring $R$, a flat Mittag-Leffler module $F$, an uncountable cardinal $\kappa$, and a subset $X$ in $F$ of cardinality $<\kappa$, there exists a pure submodule $M\subseteq F$ such that $X\subseteq M$ and $M$ is $<\kappa$-generated. Since free modules are flat Mittag-Leffler and a pure submodule of a flat Mittag-Leffler module is flat Mittag-Leffler \cite[Corollary 3.20]{GT}, this also covers the case of pure submodules of free modules settled by Osofsky \cite[Theorem I.8.10]{FS}.

 Generally speaking, however, the bound of Lemma \ref{kappa-presented-exists} is sharp.
 Indeed, let $k$ be a field of infinite cardinality $\kappa$ and $R=k[x]$ the ring of polynomials in one variable $x$ with coefficients in $k$.
 Then the spectrum of $R$ has cardinality $\kappa$, and the field of rational functions $Q=k(x)$ is a $\kappa$-generated flat $R$-module which has no nonzero proper pure submodules.
 Taking $X\subseteq Q$ to be the one-element set $X=\{1\}$, there does \emph{not} exist a $<\kappa$-generated submodule $M$ in $Q$ containing $X$.
\end{remark}

We are now ready to prove the improved deconstructibility of flat modules.

\begin{thm}\label{flats-kappa-filtered}
Let $R$ be a Noetherian commutative ring with spectrum of cardinality less than $\kappa$, where $\kappa$ is an uncountable regular cardinal. Then every flat module is a transfinite extension of $<\kappa$-generated flat modules.
\end{thm}
\begin{proof}
This is quite standard: Let $F$ be a flat module; we are going to build a filtration of $F$ by pure submodules such that the consecutive factors are $<\kappa$-generated. Let $F_0 = 0$. For every ordinal $\alpha$, pick $x \in F \setminus F_\alpha$ (if it exists, otherwise the construction is finished) and let $M$ be the $<\kappa$-generated pure submodule of the flat module $F/F_\alpha$ containing $x + F_\alpha$; this exists thanks to Lemma \ref{kappa-presented-exists}. Further let $F_{\alpha+1}$ be the preimage of $M$ in the map $F \to F/F_\alpha$; then $F_{\alpha+1}$ is a pure submodule of $F$ containing $x$ and $F_{\alpha+1}/F_\alpha \cong M$ is $<\kappa$-generated. For every limit ordinal $\alpha$, put $F_\alpha = \bigcup_{\beta<\alpha} F_\beta$. This way we exhaust the module $F$ as desired.
\end{proof}

Finally, as a special case, we obtain a new proof of~\cite[Theorem~1.17]{PS}:

\begin{cor} \label{main-corollary}
Let $R$ be a Noetherian commutative ring with countable spectrum. Then every flat module is quite flat.
\end{cor}
\begin{proof}
By Theorem \ref{flats-kappa-filtered} with $\kappa = \aleph_1$, every flat module is a transfinite extension of countably generated flat modules. By Theorem \ref{main-theorem}, countably generated flat modules are quite flat, hence all flat modules are quite flat.
\end{proof}

\section{Non-Noetherian rings} \label{section-cfq}

Let $R$ be a commutative ring. We will say that $R$ is a \emph{CFQ ring} if all countably presented flat $R$-modules are quite flat.

Recall that an associative ring $R$ is called \emph{left perfect} if all flat left $R$-modules are projective \cite{Bas}.
Obviously, all perfect commutative rings are CFQ. Theorem \ref{main-theorem} tells us that all Noetherian commutative rings are CFQ.

The following assertion is provable in the same way as Corollary \ref{countable-ordinals}: Over a CFQ ring $R$, a module $F$ is a countably presented flat module if and only if it is a direct summand of a transfinite extension, indexed by a countable ordinal, of $R$-modules of the form $S^{-1}R$, where $S$ ranges over countable multiplicative subsets of $R$.

\begin{prop} \label{zero-dimensional}
 A local ring of Krull dimension $0$ is CFQ if and only if it is perfect.
\end{prop}

\begin{proof}
In a local commutative ring $R$ of Krull dimension~$0$, every element is either invertible or nilpotent. Hence, for any multiplicative subset $S\subseteq R$, one has $S^{-1}R=R$ or $S^{-1}R=0$.  It follows that the class of quite flat $R$-modules coincides with the class of projective $R$-modules.

On the other hand, let $R$ be a local commutative ring with the Jacobson radical $J\subseteq R$. Suppose that $R$ is not perfect. Then the ideal $J$ is not T\nobreakdash-nilpotent \cite{Bas}, so there exists a sequence of elements $h_0$, $h_1$, $h_2$, \dots\ in $J$ such that the product $h_0\cdots h_n$ is nonzero for every $n\ge0$.
Consider the related Bass flat $R$-module $B$, that is, the direct limit of the sequence of $R$-module homomorphisms $R\xrightarrow{h_0}R\xrightarrow{h_1}R\xrightarrow{h_2}\cdots$.  Then $B\neq0$ is a countably presented flat $R$-module such that $JB=B$.  According to \cite[Proposition 2.7]{Bas}, $B$ is not projective.
\end{proof}

\begin{exmpl}
Let $k$ be a field, $k[x_0,x_1,x_2,\dots]$ be the ring of polynomials in a countable set of variables, and $R$ be the quotient ring of $k[x_0,x_1,x_2,\dots]$ by the ideal generated by the elements $x_i^2$, \ $i=0$, $1$, $2$,~\dots\
Then $R$ is a local commutative ring of Krull dimension $0$ with the Jacobson radical $J$ generated by the elements $x_0$, $x_1$, $x_2$,~\dots\  The ring $R$ is not perfect, since the sequence of elements $x_0$, $x_1$, $x_2$, \dots\ $\in J$ is not T\nobreakdash-nilpotent.  Hence the Bass flat $R$-module $B$ related to this sequence is not quite flat.
Notice that the ring $R$ is $<\aleph_1$-Noetherian (in fact, it can be made countable by choosing $k$ to be a countable field) and its spectrum consists of the single point $J$.
Thus the example of the ring $R$ and the flat $R$-module $B$ shows that \emph{neither} Theorem \ref{main-theorem} \emph{nor} Corollary \ref{main-corollary} holds true without the assumption of Noetherianity of the ring.
\end{exmpl}

\begin{lemma} \label{lifting-countable-flats}
Let $f\colon R\to R'$ be a homomorphism of commutative rings. Assume that for any finite sequence of elements $r'_1,\dots,r'_m\in R'$ there exist an invertible element $u'\in R'$ and a sequence of elements $r_1,\dots,r_m\in R$ such that $r'_j=u'f(r_j)$ for every $j=1,\dots,m$. Let $F'$ be a countably presented flat $R'$-module.  Then there exists a countably presented flat $R$-module $F$ such that $F'$ is isomorphic to $R'\otimes_RF$.
\end{lemma}

\begin{proof}
By \cite[Corollary~2.23]{GT}, the $R'$-module $F$ is the direct limit of a sequence of finitely generated free $R'$-modules and homomorphisms between them, indexed by the natural numbers, $P'_0\xrightarrow{h'_0} P'_1\xrightarrow{h'_1} P'_2\xrightarrow{h'_2}\cdots$.
The maps $h'_n$ are given by finite-size rectangular matrices $H'_n$ with the entries in $R'$.
By assumption, there exist invertible elements $u'_n\in R'$ and matrices $H_n$ with the entries in $R$ such that $H'_n=u'_nf(H_n)$.
Then the matrices $H_n$ define a sequence of finitely generated free $R$-modules and homomorphisms $P_0\xrightarrow{h_0} P_1\xrightarrow{h_1} P_2\xrightarrow{h_2}\cdots$ whose direct limit $F$ is the desired countably presented flat $R$-module for which $F'\cong R'\otimes_RF$.
\end{proof}

\begin{prop} \label{cfq-closed-under}
Let $R$ be a CFQ ring, $I\subseteq R$ an ideal, and $S\subseteq R$ a multiplicative subset.  Then the rings $R/I$ and $S^{-1}R$ are CFQ.
\end{prop}

\begin{proof}
By \cite[Lemma~8.3(b)]{PS}, for any commutative ring homomorphism $f\colon R\to R'$ and any quite flat $R$-module $F$, the $R'$-module $R'\otimes_RF$ is quite flat.
Now let $R'$ be one of the rings $R/I$ or $S^{-1}R$, and let $f\colon R\to R'$ be the natural homomorphism. Let $F'$ be a countably presented flat $R'$-module. By Lemma \ref{lifting-countable-flats}, there exists a countably presented flat $R$-module $F$ such that $F'\cong R'\otimes_RF$.
Since the ring $R$ is CFQ, the $R$-module $F$ is quite flat. Thus the $R'$-module $F'$ is quite flat.
\end{proof}

Let us now recall the statement of another ``Main Lemma'' from \cite{PS}, generalizing the above Lemma \ref{main-lemma} to non-Noetherian rings.
Given a multiplicative subset $S$ in a commutative ring $R$, we say that \emph{the $S$-torsion in $R$ is bounded} if there exists an element $s_0\in S$ such that for any elements $s\in S$ and $r\in R$ the equation $sr=0$ in $R$ implies $s_0r=0$.

\begin{lemma}[{\cite[Main Lemma 1.23]{PS}}] \label{main-lemma2}
Let $R$ be a commutative ring and $S\subseteq R$ be a countable multiplicative subset such that the $S$-torsion in $R$ is bounded.  Then a flat $R$-module $F$ is quite flat if and only if the $R/sR$-module $F/sF$ is quite flat for all $s\in S$ and the $S^{-1}R$-module $S^{-1}F$ is quite flat.
\end{lemma}

\begin{prop}
Let $R$ be a commutative ring and $S\subseteq R$ be a countable multiplicative subset such that the $S$-torsion in $R$ is bounded. Assume that the ring $S^{-1}R$ is CFQ and, for every element $s\in S$, the ring $R/sR$ is CFQ.  Then the ring $R$ is CFQ.
\end{prop}

\begin{proof}
Follows immediately from Lemma \ref{main-lemma2}.
\end{proof}

\begin{thm} \label{when-domain-cfq}
Let $R$ be an integral domain. Then $R$ is CFQ if and only if for every nonzero element $s\in R$, the ring $R/sR$ is CFQ.
\end{thm}

\begin{proof}
The implication ``only if'' is provided by Proposition \ref{cfq-closed-under}.  Let us prove the ``if''. Let $F$ be a countably presented flat $R$-module. Denote by $S=R\setminus\{0\}$ the multiplicative subset of all nonzero elements in $R$. Then the $S^{-1}R$-module $S^{-1}F$ is projective, since $S^{-1}R$ is a field.  By Proposition \ref{countable-subset}, there exists a countable multiplicative subset $T\subseteq S$ such that the $T^{-1}R$-module $T^{-1}F$ is projective.

For every element $t\in T$, the $R/tR$-module $F/tF$ is quite flat, since it is a countably presented flat module and the ring $R/tR$ is CFQ.  Furthermore, the $T$-torsion in $R$ is bounded (in fact, zero), since $R$ is a domain. By Lemma \ref{main-lemma2}, it follows that the $R$-module $F$ is quite flat.
\end{proof}

The next theorem is a common generalization of Theorem \ref{when-domain-cfq} and of the induction step in the proof of Theorem \ref{main-theorem}.

\begin{thm} \label{bounded-torsion-perfect-localization}
Let $R$ be a commutative ring and $S\subseteq R$ be a multiplicative subset such that the $S$-torsion in $R$ is bounded. Assume that the ring $S^{-1}R$ is perfect and, for every element $s\in S$, the ring $R/sR$ is CFQ.  Then the ring $R$ is CFQ.
\end{thm}

\begin{proof}
Let $F$ be a countably presented flat $R$-module.  Then the $S^{-1}R$-module $S^{-1}F$ is projective, since the ring $S^{-1}R$ is perfect.  By Proposition \ref{countable-subset}, there exists a countable multiplicative subset $T_0\subseteq S$ such that the $T_0^{-1}R$-module $T_0^{-1}F$ is projective.
Furthermore, by the assumption of bounded $S$-torsion in $R$ there exists an element $s_0\in S$ annihilating all the $S$-torsion in $R$.

Let $T\subseteq S$ be the multiplicative subset generated by $T_0$ and $s_0$.  Then $T$ is also countable and the $T^{-1}R$-module $T^{-1}F$ is projective, but in addition the $T$-torsion in $R$ is bounded (by $s_0$).  Finally, for any element $t\in T$ the $R/tR$-module $F/tF$ is quite flat, since it is countably presented flat and the ring $R/tR$ is CFQ.  By Lemma \ref{main-lemma2}, it follows that the $R$-module $F$ is quite flat.
\end{proof}

The next lemma and proposition provide another approach to the CFQ property of non-domains.

\begin{lemma} \label{zerodivisor-pair-quiteflatness}
Let $R$ be a commutative ring and $a$, $b\in R$ be a pair of elements for which $ab=0$. Let $F$ be a flat $R$-module such that the $R/aR$-module $F/aF$ is quite flat and the $R/bR$-module $F/bF$ is quite flat.  Then the $R$-module $F$ is quite flat.
\end{lemma}

\begin{proof}
Let $C$ be an almost cotorsion $R$-module. We have to prove that $\Ext^1_R(F,C)=0$.
Notice that an $R/aR$-module is almost cotorsion if and only if it is almost cotorsion as an $R$-module \cite[Lemma~8.4]{PS}, and any quotient module of an almost cotorsion module is almost cotorsion \cite[Lemma~8.1(a)]{PS}.

Consider the short exact sequence of $R$-modules $0\to aC\to C\to C/aC\to 0$.  Then $C/aC$ is an $R/aR$-module.  It is also a quotient $R$-module of $C$; so it is an almost cotorsion $R/aR$-module.  Similarly, $aC$ is an $R/bR$-module.  It is a quotient $R$-module of $C$ as well, so it is an almost cotorsion $R/bR$-module.

Now $\Ext^1_R(F,C/aC)=\Ext^1_{R/aR}(F/aF,C/aC)$ by \cite[Lemma 4.3]{PS} and the Ext group in the right-hand side vanishes, since the $R/aR$-module $F/aF$ is quite flat and the $R/aR$-module $C/aC$ is almost cotorsion.
Similarly, $\Ext^1_R(F,aC)=\Ext^1_{R/bR}(F/bF,aC)$ by \cite[Lemma 4.3]{PS} and the latter Ext group vanishes, since the $R/bR$-module $F/bF$ is quite flat and the $R/bR$-module $aC$ is almost cotorsion.
In view of the above short exact sequence, we conclude that $\Ext^1_R(F,C)=0$.
\end{proof}

\begin{prop} \label{zerodivisor-pair-CFQ}
Let $R$ be a commutative ring and $a$, $b\in R$ be a pair of elements for which $ab=0$.  Assume that both the rings $R/aR$ and $R/bR$ are CFQ. Then the ring $R$ is CFQ.
\end{prop}

\begin{proof}
Follows immediately from Lemma \ref{zerodivisor-pair-quiteflatness}.
\end{proof}

We recall that a commutative integral domain $R$ is called \emph{almost perfect} \cite{BS,Sal} if for every nonzero element $s\in R$ the ring $R/sR$ is perfect.
More generally, let $S$ be a multiplicative subset in a commutative ring $R$.  Then the ring $R$ is said to be \emph{$S$-almost perfect} \cite{BP} if the ring $S^{-1}R$ is perfect and, for every element $s\in S$, the ring $R/sR$ is perfect.

\begin{prop} \label{S-almost-perfect}
 Let $R$ be an $S$-almost perfect commutative ring. Then $R$ is CFQ.
\end{prop}

\begin{proof}
Let $F$ be a countably presented flat $R$-module. Then the $S^{-1}R$-module $S^{-1}F$ is projective, since the ring $S^{-1}R$ is perfect. By Proposition \ref{countable-subset}, there exists a countable multiplicative subset $T\subseteq S$ such that the $T^{-1}R$-module $T^{-1}F$ is projective.

Now for every element $t\in T$, the $R/tR$-module $F/tF$ is projective, since the ring $R/tR$ is perfect.  By \cite[Theorem 1.3]{PS}, it follows that the $R$-module $F$ is even \emph{$T$-strongly flat}, i.e., $F$ is a direct summand of an $R$-module $G$ for which there is a short exact sequence of $R$-modules $0\to U\to G\to V\to 0$, where $U$ is a free $R$-module and $V$ is a free $T^{-1}R$-module.
In particular, $F$ is quite flat.
\end{proof}

\begin{cor} \label{almost-perfect-local-domains}
Let $R$ be a local integral domain of Krull dimension~$1$.  Then $R$ is CFQ if and only if it is an almost perfect domain.
\end{cor}

\begin{proof}
If $R$ is almost perfect, then it is CFQ either by Theorem \ref{when-domain-cfq} or by Proposition \ref{S-almost-perfect}.  Conversely, if $R$ is CFQ, then the ring $R/sR$ is CFQ
for every $s\in R$ by Proposition \ref{cfq-closed-under}.  When $s\ne0$, the ring $R/sR$ is a zero-dimensional local CFQ ring, so it is perfect by Proposition \ref{zero-dimensional}.
\end{proof}

\begin{exmpl}\label{puiseux}
Let $k$ be a field and $R=k[x,x^{1/2},x^{1/3},\dots]$ be the ring of Puiseux series with the coefficients in $k$.  Then $R$ is a one-dimensional local domain which is not almost perfect.  Indeed, $R$ is a non-discrete valuation domain, while every almost perfect valuation domain is a DVR \cite[Example~3.2]{Sal}.  Besides, the intersection $\bigcap_{n\ge0}\p^n$ of all powers of the maximal ideal $\p$ in $R$ coincides with $\p$, while one has $\bigcap_{n\ge0}\p^n=0$ in any almost perfect local domain \cite[Corollary 4.2]{Sal}.
Another example of a one-dimensional local domain that is not almost perfect can be found in \cite[Example 1.3]{Zan}.  In view of Corollary \ref{almost-perfect-local-domains}, these are examples of non-CFQ domains.  As the spectrum of a one-dimensional local domain consists of two points, these examples also show that Corollary \ref{main-corollary} does not hold for non-Noetherian domains.
Moreover, the ring of Puiseux series is a non-CFQ coherent domain.
\end{exmpl}

More generally, by \cite[Proposition 4.6]{BS}, a coherent domain is almost perfect if and only if it is Noetherian of Krull dimension~$1$.
Hence a one-dimensional local coherent domain is CFQ if and only if it is Noetherian.

It follows from Corollary \ref{almost-perfect-local-domains} and Proposition \ref{cfq-closed-under} that every one-dimensional CFQ domain is \emph{locally almost perfect}, i.e., its localizations at its maximal ideals are almost perfect.
We do \emph{not} know whether all locally almost perfect domains are CFQ (cf.\ Example \ref{cfq-Bezout-example} below).

\begin{thm} \label{vNRs-are-CFQ}
All von Neumann regular commutative rings are CFQ.
\end{thm}

\begin{proof}[First proof]
Indeed, let $R$ be a von Neumann regular commutative ring; so for every element $a\in R$ there exists $b\in R$ such that $aba=a$. Then the principal ideal $Ra$ generated by the element $a$ in $R$ coincides with the ideal $Rab$ generated by the idempotent element $ab$; and the localization $R[a^{-1}]$ of the ring $R$ at the multiplicative subset generated by $a$ coincides with the localization $R[(ab)^{-1}]$ at the multiplicative subset generated by $ab$.  For an idempotent element $e\in R$, one has $R[e^{-1}]=R/(1-e)R$. Thus the localizations of $R$ at countable multiplicative subsets are the same thing as the quotient rings of $R$ by countably generated ideals.

Over a von Neumann regular ring, all modules are flat.  Furthermore, the ring $R$ is coherent.  Hence any countably generated submodule of a countably presented $R$-module is countably presented.

Now let $F$ be countably presented $R$-module and let $\{f_n\mid n\in\N_0\}$ be its countable set of generators.  Then for any $n\ge0$ the $R$-module $F_n=F/(Rf_0+\cdots+Rf_n)$ is countably presented.  Let $G_0=Rf_0$ be the cyclic submodule generated by $f_0$ in $F$ and $G_{n+1}$ be the cyclic submodule generated by the coset of the element $f_{n+1}$ in $F_n$.
Then $G_n$ is a countably generated submodule of a countably presented $R$-module, hence $G_n$ is countably presented.  Being cyclic, $G_n$ is isomorphic to
the quotient of $R$ by a countably generated ideal. The $R$-module $F$ is filtered by the $R$-modules $G_0$, $G_1$, $G_2$,~\dots\  Thus $F$ is quite flat.
\end{proof}

\begin{proof}[Second proof]
More generally, by~\cite[Corollary~2.23]{GT}, a countably presented flat module $F$ over a commutative ring $R$ can be described by a sequence of finite matrices $h_0$, $h_1$, $h_2$,~\dots\ with entries in $R$ (as in the proof of Lemma~\ref{lifting-countable-flats}). All entries of such a sequence of matrices $(h_n)_{n\ge0}$ form a countable set of elements in $R$.  Let $\overline{R}\subseteq R$ be a subring containing all these matrix entries.  Then there is a countably presented flat $\overline{R}$-module $\overline{F}$
such that $F=R\otimes_{\overline{R}}\overline{F}$.

Now let $R$ be a von Neumann regular commutative ring and $R_0\subseteq R$ be a countable subring.  Define inductively a sequence of subrings $R_0\subseteq R_1\subseteq R_2\subseteq\cdots$ in $R$ indexed by the natural numbers $n\in\N_0$ as follows.  For every element $a\in R_n$, choose an element $b\in R$ such that $aba=a$.  Denote by $R_{n+1}\subseteq R$ the subring generated by $R_n$ and all the elements $b$ so chosen.  Put $\overline{R}=\bigcup_{n\ge0}R_n$.  Then $\overline{R}$ is a countable von Neumann regular subring in $R$ containing $R_0$.

Let $F$ be a countably presented module over a von Neumann regular commutative ring $R$.  Using the previous observations, one can construct a countable von Neumann regular subring $\overline{R}\subseteq R$ and a countably presented $\overline{R}$-module $\overline{F}$ such that $F=R\otimes_{\overline{R}}\overline{F}$.  Now we observe that every $\overline{R}$-module is filtered by cyclic $\overline{R}$-modules, and all cyclic $\overline{R}$-modules $\overline{R}/\overline{I}$ are quotients of $\overline{R}$ by countably generated ideals $\overline{I}$.
According to the first proof, since the ring $\overline{R}$ is von Neumann regular, we have $\overline{R}/\overline{I}=\overline{S}^{-1}\overline{R}$ for a certain countable multiplicative subset $\overline{S}\subseteq\overline{R}$.  Thus all $\overline{R}$-modules are quite flat.  By \cite[Lemma 8.3(b)]{PS}, the $R$-module $F=R\otimes_{\overline{R}}\overline{F}$ is quite flat.
\end{proof}

\begin{remark} \label{countable-vN-regular}
The assertion of Corollary \ref{main-corollary} also holds true with the Noetherianity condition replaced by the von Neumann regularity condition.  Moreover, similarly to \cite[Remark 8.10]{PS}, for any von Neumann regular commutative ring $R$ with countable spectrum there exists a countable collection of countable multiplicative subsets $S_1$, $S_2$, $S_3$,~\dots\ $\subseteq R$ such that every $R$-module is filtered by modules isomorphic to $S_j^{-1}R$, \ $j=1$, $2$, $3$,~\dots\  Indeed, the spectrum of a von Neumann regular ring $R$ is a compact Hausdorff space, and ideals in $R$ correspond bijectively to closed subsets of the spectrum.
By Baire's category theorem, any countable compact Hausdorff space has an isolated point.  It follows that any von Neumann regular ring $R$ with countable spectrum is semiartinian, i.e., all $R$-modules are filtered by simple $R$-modules.  It remains to let the index $j$ number the points of $\Spec R$, and observe that all the prime ideals $\p_j\in\Spec R$ are countably generated, so there exists a countable multiplicative subset $S_j\subseteq R$ such that $S_j^{-1}R=R/\p_j$.
\end{remark}

\begin{exmpl} \label{cfq-Bezout-example}
Here is an example of a non-almost perfect one-dimensional CFQ domain \cite[Example 3.7]{Sal}.  The domain in question is a B\'ezout ring \cite[Section~III.5]{FS}, that is, a ring in which every finitely generated ideal is principal.  The divisibility group of a B\'ezout domain is a lattice-ordered group, and conversely, any lattice-ordered group is the divisibility group of a B\'ezout domain.

We are interested in a B\'ezout domain $R$ whose divisibility group $\Gamma$ is isomorphic to the subgroup of all eventually constant sequences of integers in $\mathbb Z^\N$ with pointwise ordering.  Following \cite[Example 3.7]{Sal}, all the localizations of $R$ at its maximal ideals are Noetherian discrete valuation rings, still $R$ is not Noetherian and not h-local, hence not almost perfect.

Let us show that $R$ is a CFQ ring.  For every $n\in\N$, consider the valuation $v_n$ on $R$ corresponding to the $n$-th coordinate in $\Gamma$.  For any nonzero element $s\in R$, consider an element $a\in R$ such that $v_n(a)=0$ whenever $v_n(s)=0$ and $v_n(a)=1$ whenever $v_n(s)>0$.  Then the nilradical of the ring $R/sR$ is the principal ideal generated by a nilpotent element~$a+sR$.

The quotient ring $R/aR$ is von Neumann regular, hence CFQ by Theorem~\ref{vNRs-are-CFQ}.  Applying Proposition~\ref{zerodivisor-pair-CFQ} iteratively to the rings $R/a^mR$, $m\ge1$, we conclude that the ring $R/sR$ is CFQ.  By Theorem~\ref{when-domain-cfq}, the domain $R$ is CFQ.

Besides, commutative von Neumann regular rings whose spectrum has no isolated points are examples of zero-dimensional CFQ rings which are not $S$\nobreakdash-almost perfect.
\end{exmpl}

\subsection{Valuation domains} \label{section-valuation-domains}
	Recall that an integral domain $R$ is a \emph{valuation domain} if its lattice of ideals is totally ordered, and that $R$ is a \emph{Pr\"{u}fer domain} if $R_{\p}$ is a valuation domain for each $\p \in \Spec(R)$. Any valuation domain is a Pr\"{u}fer domain. Furthermore, all Pr\"{u}fer domains have weak global dimension at most one (\cite[Corollary 4.2.6]{Gl}), meaning that submodules of flat $R$-modules are flat.
	
	The theory of purity simplifies considerably over Pr\"{u}fer domains, which will be useful to recall for the sequel. Let $R$ be a Pr\"{u}fer domain and $Q$ its field of quotients. By Warfield's theorem \cite[Theorem I.8.11]{FS}, it is sufficient to check purity over $R$ on simple divisibility equations of the form $rx = a$, where $r \in R$. As a simple consequence, flat $R$-modules coincide with the $R$-modules which are torsion-free. Let $F$ be a flat $R$-module and $M$ a submodule of $F$. Then the equation $rx = a$ has at most one solution in $F$ for each $r \in R$ and $a \in F$. Therefore, the intersection of all pure submodules of $F$ containing $M$ is a pure submodule of $F$, called the \emph{purification} of $M$ in $F$, see \cite[p. 47]{FS}. Clearly, the purification of $M$ in $F$ is of the form $\langle M \rangle_* = \{f \in F \mid rf \in M \text{ for some non-zero $r \in R$}\}$.
    Given an $R$-module $N$, we say that $N$ is of \emph{rank $\kappa$}, where $\kappa$ is a cardinal number, if the vector space $N \otimes_R Q$ is of dimension $\kappa$ over $Q$. It follows directly from the description of purification above that if $M$ is a submodule of a flat $R$-module $F$, then the ranks of $M$ and $\langle M \rangle_*$ are the same.

    If $R$ is a valuation domain, then all the localizations of $R$ at multiplicative subsets are localizations at prime ideals, i.e., for every multiplicative subset $S\subseteq R$ there exists a prime ideal $\q$ in $R$ such that $S^{-1}R=R_\q$ \cite[Proposition II.1.5]{FS}.

\begin{lemma}\label{prufer-rankone}
	Let $R$ be a Pr\"{u}fer domain. Then $R$ is a CFQ ring if and only if every countably generated flat $R$-module of rank $1$ is quite flat.
\end{lemma}
\begin{proof}
	First, it follows from \cite[Proposition 6]{EFS} that all countably generated flat $R$-modules are countably presented. This renders the only-if part of the statement trivial. Let us prove the other implication. Let $F$ be a countably generated flat $R$-module with some fixed set $\{x_n \mid n \in \N\}$ of generators. For each $n \in \N_0$, let $M_n$ be the submodule of $F$ generated by the elements $\{x_1,x_2,\dots,x_n\}$ and let $F_n = \langle M_n \rangle_*$ be the purification of $M_n$ in $F$. This yields a pure filtration $F = \bigcup_{n \in \N_0} F_n$ of $F$ such that the consecutive quotients $F_{n+1}/F_n$ for $n \in \N_0$ are flat $R$-modules of rank $1$. For each $n \in \N_0$, the flat $R$-module $F/F_{n+1}$ is countably generated, and therefore countably presented. It follows that $F_{n+1}$ is a countably generated $R$-module, and therefore $F_{n+1}/F_n$ is a countably generated flat $R$-module of rank $1$ for each $n \in \N_0$. By the assumption, $F_{n+1}/F_n$ is quite flat for each $n \in \N_0$, and since $F$ is filtered by these, $F$ is quite flat.
\end{proof}

Following \cite[\S II.8 and \S III.7]{FS}, we call a Pr\"{u}fer domain $R$ \emph{strongly discrete} if no non-zero prime ideal of $R$ is idempotent.
A Pr\"{u}fer domain is strongly discrete if and only if all its localizations at prime ideals are strongly discrete valuation domains \cite[Proposition III.7.4]{FS}. Before going into the proof of the main result of this subsection, let us recall from \cite[Lemmas II.4.3(iv) and II.4.4]{FS} that for any prime ideal $\p$ of a valuation domain $R$ we have $\p = \p R_{\p}$. In particular, we can view any non-zero prime ideal $\p$ of $R$ as the maximal ideal of the valuation domain $R_{\p}$.

\begin{thm}\label{valuation-domains}
	Let $R$ be a valuation domain. Then $R$ is CFQ if and only if $R$ is strongly discrete.
\end{thm}
\begin{proof}
	We start with the assumption that $R$ is a strongly discrete valuation domain. By Lemma~\ref{prufer-rankone}, it is enough to show that any countably generated flat $R$-module of rank $1$ is quite flat.
	
	We claim that for any flat $R$-module $I$ of rank $1$ there is a prime ideal $\p \in \Spec(R)$ such that $I \cong R_{\p}$. Indeed, since $I$ is of rank $1$, we can view $I$ as a submodule of $Q$. If $I = Q$, the claim is true for $\p = 0$, so we can further assume $I \neq Q$. Then for any $q \in Q \setminus I$ we have $q^{-1}I \subseteq R$, and therefore we can assume that $I$ is an ideal of $R$, see also \cite[Lemma II.1.4]{FS}. From the strong discreteness of $R$ and \cite[Theorem II.8.3]{FS}, we then obtain that $I$ is isomorphic to a (necessarily non-zero) prime ideal $\p$ of $R$.  Moreover, since $\p$ is not idempotent, $\p$ is then a principal ideal of the ring $R_{\p}$ (\cite[p. 69(d)]{FS}), and therefore $\p \cong R_{\p}$ as $R$-modules, validating the claim.
	
	Now if $I$ is countably generated, then there is a countable multiplicative subset $S$ of $R$ such that $I \cong R_{\p} \cong S^{-1}R$. Therefore, $I$ is quite flat.

	Next we aim to prove the converse implication, so let us assume that $R$ is CFQ. First, let $\p \subseteq \q$ be prime ideals in $R$ such that $\p \neq \q$, and such that there is no prime ideal between $\p$ and $\q$ in $(\Spec(R),\subseteq)$. Then the domain $R_{\q}/\p$ is a valuation domain of Krull dimension one, and it is CFQ by Proposition~\ref{cfq-closed-under}. By Corollary~\ref{almost-perfect-local-domains}, $R$ is an almost perfect domain, and therefore using Example~\ref{puiseux} we get that the maximal ideal $\q/\p$ of $R_{\q}/\p$ cannot be idempotent. It follows that $\q$ is not an idempotent ideal of the ring $R$. We proved that all prime ideals of $R$ which are successors in $(\Spec(R),\subseteq)$ are not idempotent.

	We finish the proof by showing that $R$ is not CFQ if the totally ordered set $(\Spec(R),\subseteq)$ does not satisfy the ascending chain condition (cf. \cite[Theorem II.8.3]{FS}). In such a case, there is a strictly increasing chain $\p_0 \subset \p_1 \subset \p_2 \subset \cdots$ of prime ideals in $\Spec(R)$ indexed by $\N_0$, and we denote the limit prime ideal as $\m = \bigcup_{n \in \N_0}\p_n$. Using Proposition~\ref{cfq-closed-under} again, it is enough to show that the valuation domain $R_{\m}$ is not CFQ, and so we can assume that $\m$ is the maximal ideal of $R$. Note that $\m$ is generated by the set $\{s_n \mid n > 0\}$ of elements of $R$, where $s_n$ is any element from $\p_{n+1} \setminus \p_n$. Therefore, $\m$ is a countably generated flat $R$-module of rank $1$. In view of Lemma~\ref{prufer-rankone}, it is sufficient to show that $\m$ is not quite flat. We proceed by the following inductive argument.

	By transfinite induction on ordinal $\lambda$, we show that $\m$ is not isomorphic to a direct summand in an $R$-module $F$ which admits a filtration $F = \bigcup_{\alpha < \lambda} F_\alpha$, where $F_{\alpha+1}/F_\alpha$ is isomorphic to $R_{\q_\alpha}$ for some prime ideal $\q_\alpha \in \Spec(R)$ (see the note preceding Lemma \ref{prufer-rankone}) for each $\alpha < \lambda$. The case of $\lambda = 0$ is clear as $F_0 = 0$. Assume first that $\lambda$ is a limit ordinal. Since $\m \subseteq \bigcup_{\alpha < \lambda} F_\alpha$, there is $\beta < \lambda$ such that $\m \cap F_\beta \neq 0$. But $\m \cap F_\beta$ is a pure submodule of $\m$, and since $\m$ is of rank $1$, this necessarily means that $\m \subseteq F_\beta$. Therefore, $\m$ is a direct summand in $F_\beta = \bigcup_{\alpha < \beta + 1} F_\alpha$, which is a contradiction by the induction hypothesis. 

Finally, let as assume that $\lambda$ is a non-limit ordinal, and write $\lambda = \beta + m$, where $\beta$ is either a limit ordinal or zero, and $m > 0$ is a positive natural number. Because $\m = \bigcup_{n \in \N_0} \p_n$, and $\Spec(R)$ is totally ordered, there is $k \in \N_0$ such that $\q_{\beta+i}$ is properly contained in $\p_k$ for all $i=0,1,\ldots,m-1$ such that $\q_{\beta+i} \neq \m$. It follows that $(R/\p_k) \otimes_R R_{\q_{\beta+i}}$ is a projective $R/\p_k$-module for all $i=0,1,\ldots,m-1$.
In fact, one has $(R/\p_k) \otimes_R R_{\q_{\beta+i}}\cong R/\p_k$ or $0$.
Consequently, the $R/\p_k$-module $(F/F_{\beta}) \otimes_R (R/\p_k)$ is projective, and therefore \begin{equation}\label{splitting}
    F \otimes_R (R/\p_k) \cong \bigl(F_\beta \otimes_R (R/\p_k)\bigr) \oplus (R/\p_k)^{(l)}
\end{equation} 
for some $0 \leq l \leq m$. Also, we have $\m\p_k=\p_k$, and therefore $\m/\p_k\cong\m\otimes_R (R/\p_k)$. Then the maximal ideal $\m/\p_k$ of the valuation domain $R/\p_k$ is again isomorphic to a direct summand in $F \otimes_R (R/\p_k)$. Note that $\m/\p_k$ can be written as the union of the strictly increasing chain $0 = \p_k/\p_k \subset \p_{k+1}/\p_k \subset \p_{k+2}/\p_k \subset \cdots$ of prime ideals of $R/\p_k$, and it is an idempotent ideal.

If $\beta = 0$, then $F_\beta = 0$, and therefore $\m/\p_k$ is a projective $R/\p_k$-module. But then $\m/\p_k$ is a principal ideal in the valuation domain $R/\p_k$, which is impossible since $\m/\p_k$ is a non-zero idempotent ideal. If $\beta$ is a limit ordinal, the isomorphism (\ref{splitting}) allows us to rearrange the terms of the filtration in order to show that $F \otimes_R (R/\p_k)$ admits a filtration by localizations of the valuation domain $R/\p_k$ indexed by the limit ordinal $l + \beta = \beta$. Therefore, we conclude that $\m/\p_k$ being a direct summand in $F \otimes_R (R/\p_k)$ is in contradiction with the induction premise for the ordinal $\beta < \lambda$ applied in the case of the valuation domain $R/\p_k$.
\end{proof}

\begin{cor}
	Let $R$ be a Pr\"{u}fer domain. If $R$ is CFQ then $R$ is strongly discrete.
\end{cor}
\begin{proof}
	Follows by combining Theorem~\ref{valuation-domains} and Proposition~\ref{cfq-closed-under}.
\end{proof}

\begin{remark}
	If $R$ is a strongly discrete valuation domain, then the totally ordered set $(\Spec(R), \supseteq)$ satisfies the descending chain condition (\cite[Theorem II.8.3]{FS}), and therefore is order-isomorphic to an ordinal number. On the other hand, any ordinal number is order isomorphic to $(\Spec(R) \setminus \{0\}, \supseteq)$ for a suitable strongly discrete valuation domain, see \cite[Example II.8.5]{FS}.
\end{remark}

\section{Finitely quite flat modules}
\label{section-finitely-quite-flat}

Let $R$ be a commutative ring and $S_1$, \dots, $S_m\subseteq R$ be a finite sequence of multiplicative subsets. For brevity, we will denote the collection of multiplicative subsets $S_1$,~\dots, $S_m$ by a single letter $\mathbf S$. A left $R$-module $F$ is said to be \emph{$\mathbf S$-strongly flat} if it is a direct summand of an $R$-module filtered by $R$, $S_1^{-1}R$, \dots, $S_m^{-1}R$.
 
For any subset of indices $L\subseteq\{1,\dots,m\}$, denote by $S_L\subseteq R$ the multiplicative subset in $R$ generated by (the union of) the multiplicative subsets $S_l$, $l\in L$.  Let $\mathbf S^\times$ denote the collection of $2^m$ multiplicative subsets $S_L\subseteq R$, \ $L\subseteq\{1,\dots,m\}$.
 
Put $K=\{1,\dots,m\}\setminus L$. For every $k\in K$, choose an element $s_k\in S_k$, and denote the collection of elements $(s_k)_{k\in K}$ by $\mathbf s$. Let $R_{L,\mathbf s}$ denote the quotient ring of the ring $S_L^{-1}R$ by the ideal generated by the elements $s_k$, $k\in K$.  So the ring $R_{L,\mathbf s}$ is obtained from the ring $R$ by inverting all the elements of the multiplicative subsets $S_l$, $l\in L$, and annihilating one chosen element $s_k\in S_k$ for every $k\in K$.

The following theorem is a particular case of \cite[Theorem 1.10]{PS}.  It only differs from the general case in that we assume all the multiplicative subsets to be countable.

\begin{thm} \label{s-times-strongly-flat}
Let $R$ be a commutative ring and $S_1$,~\dots, $S_m\subseteq R$ be a finite sequence of (at most) countable multiplicative subsets in $R$.  Let $F$ be a flat $R$-module.
Then the $R$-module $F$ is $\mathbf S^\times$-strongly flat if and only if the $R_{L,\mathbf s}$-module $R_{L,\mathbf s}\otimes_R F$ is projective for every subset of indices $L\subseteq\{1,\dots,m\}$ and any choice of elements $s_k\in S_k$, $k\in\{1,\dots,m\}\setminus L$.
\end{thm}

\begin{exmpl}
	In \cite[Example II.8.6]{FS}, a valuation domain $R$ is constructed such that its spectrum, as a totally ordered set, is of the following form
	$$0 = \p_0 \subset \p_1 \subset \p_2 \subset \p_3 \subset \cdots \subset \m = \bigcup_{n \in \N_0} \p_n,$$
	and such that the prime ideals $\p_n$ are not idempotent for all $n \in \N$. The maximal ideal $\m$ is necessarily idempotent, and therefore $R$ is not strongly discrete, but it is discrete in the terminology of \cite[\S II.8]{FS}. For each $n \in \N_0$, there is a countable multiplicative subset $S_n$ such that $R_{\p_n} = S_n^{-1}R$ --- indeed, we can let $S_n$ be the multiplicative subset generated by any element $t_n \in \p_{n+1} \setminus \p_n$ for all $n \in \N$. Note that the $R_{\p_n}$-module $(\m \otimes_R R_{\p_n}) \cong R_{\p_n}$ is projective for each $n \in \N$. Furthermore, for any choice of element $s_n \in S_n \setminus \{1\}$, $n \in \N$, the set $\{s_n \mid n \in \N\}$ always generates the maximal ideal $\m$. Therefore, $(R/\m) \otimes_R \m = \m/\m^2 = 0$, a projective module again. Together, these conditions on projectivity of localizations and quotients of $\m$ yield the hypothesis of Theorem \ref{s-times-strongly-flat}, but generalized in a na\"{\i}ve way from finitely many multiplicative subsets to countably many. On the other hand, as demonstrated in the proof of Theorem~\ref{valuation-domains}, the (countably presented) flat $R$-module $\m$ is not quite flat. This shows that the natural na\"{\i}ve generalization of the statement of Theorem \ref{s-times-strongly-flat} from finitely many to countably many multiplicative subsets is no longer valid.
\end{exmpl}

A commutative ring is called \emph{locally perfect} if all its localizations at the maximal ideals are perfect.  A commutative ring $R$ is locally perfect if and only if its Jacobson radical $J$ is T-nilpotent and the quotient ring $R/J$ is von Neumann regular \cite{Fai}.

The following result is certainly not new, but we are not aware of a suitable reference.

\begin{prop} \label{projective-reduction}
Let $R$ be a (not necessarily commutative) ring and $I\subset R$ be a left T\nobreakdash-nilpotent two-sided ideal.  Then a flat left $R$-module $F$ is projective if and only if the $R/I$-module $F/IF$ is projective.
\end{prop}

\begin{proof}
Clearly, if $F$ is projective over $R$, then $F/IF$ is projective over $R/I$.  In order to prove the converse, we first consider the case when $F/IF$ is a free $R/I$-module; so $F/IF\cong (R/I)^{(X)}$ for a certain set $X$.

Let $G=R^{(X)}$ be the free left $R$-module with a basis indexed by the same set $X$.  Then we have $G/IG\cong F/IF$, and the obvious surjective $R$-module morphism $G\to F/IF$ can be lifted to an $R$-module morphism $g\colon G\to F$.

Let $C$ be the cokernel of $g$.  Then $C/IC=0$, and by \cite[Lemma 28.3]{AF} it follows that $C=0$.  Hence the map $g$ is surjective.

Let $K$ denote the kernel of $g$; then we have a short exact sequence of left $R$-modules $0\to K\to G\to F\to 0$.  Since $F$ is flat, tensoring by $R/I$ gives the short exact sequence $0\to K/IK\to G/IG\to F/IF\to0$. The map $G/IG\to F/IF$ is an isomorphism by construction, so $K/IK=0$.  Applying \cite[Lemma 28.3]{AF} again, we conclude that $K=0$ and $F\cong G$.

The general case follows by Eilenberg's trick.  Suppose $F/IF$ is a direct summand of a free $R/I$-module $(R/I)^{(Y)}$.  Denote by $Z$ the set $Y\times\N$.  Then the $R/I$-modules $(R/I)^{(Z)}$ and $F/IF\oplus (R/I)^{(Z)}$ are both free (and isomorphic to each other).  Consider the $R$-module $F'=F\oplus R^{(Z)}$.  The $R/I$-module $F'/IF'$ is free and the $R$-module $F'$ is flat, so it follows that $F'$ is a free $R$-module, as we have already proved.  Thus $F$ is a projective $R$-module.
\end{proof}

By analogy with the discussion of ``finitely very flat modules'' in \cite{PS0}, let us define finitely quite flat modules.  A module $F$ over a commutative ring $R$ is \emph{finitely quite flat} if there exists a finite collection of countable multiplicative subsets $S_1$,~\dots, $S_m\subset R$ such that $F$ is a direct summand of an $R$-module filtered by modules isomorphic to $R$, $S_1^{-1}R$,~\dots, $S_m^{-1}R$ (i.e., in other words, $F$ is $\mathbf S^\times$-strongly flat).
Obviously, any finitely quite flat module is quite flat.

The following theorem is our motivation for considering finitely quite flat modules.  It would be interesting to know whether it holds true for quite flat modules instead of finitely quite flat ones.

\begin{thm} \label{finitely-quite-flat-reduction}
Let $R$ be a commutative ring and $I\subset R$ be a T\nobreakdash-nilpotent ideal.  Then a flat $R$-module $F$ is finitely quite flat if and only if the $R/I$-module $F/IF$ is finitely quite flat.
\end{thm}

\begin{proof}
For any commutative ring homomorphism $R\to R'$ and any finitely quite flat $R$-module $F$, the $R'$-module $R'\otimes_RF$ is finitely quite flat (cf.~\cite[Lemma 2.2(b)]{PS0}).  Hence the ``only if'' implication is clear.

To prove the ``if'', consider a collection of countable multiplicative subsets $S_1'$,~\dots, $S_m'\subset R/I$ such that the $R/I$-module $F/IF$ is a direct summand of an $R/I$-module filtered by (modules isomorphic to) $R/I$, $S_1^{\prime-1}(R/I)$,~\dots, $S_m^{\prime-1}(R/I)$.  Arguing as in \cite[Lemma 8.4]{PS}, we lift the multiplicative subsets $S_l'\subset R/I$, $1\leq l\leq m$, to countable multiplicative subsets $S_l\subset R$.

Put $R'=R/I$ and $F'=F/IF$.  Denote the collection of multiplicative subsets $S_1'$,~\dots, $S_m'\subseteq R'$ by $\mathbf S'$ and the collection of multiplicative subsets $S_1$,~\dots, $S_m\subseteq R$ by $\mathbf S$. The $R'$-module $F'$ is $\mathbf S'$-strongly flat, hence also $\mathbf S^{\prime\times}$-strongly flat.  By Theorem \ref{s-times-strongly-flat}, it follows that the $R'_{L,\mathbf s'}$-module $F'_{L,\mathbf s'}=R'_{L,\mathbf s'}\otimes_{R'}F'$ is projective for every subset $L\subseteq\{1,\dots,m\}$ and any choice of elements $s_k'\in S_k'$, $k\in K=\{1,\dots,m\}\setminus L$ (where $\mathbf s'$ denotes the collection of chosen elements $(s_k')_{k\in K}$).

Let $s_k\in S_k$, $k\in K$ be some elements and $s_k'\in S_k'$ be their images under the surjective ring homomorphism $R\to R'$. Then there is a natural surjective ring homomorphism $f_{L,\mathbf s}\colon R_{L,\mathbf s}\to R'_{L,\mathbf s'}$ whose kernel is generated by the image of $I$ in $R_{L,\mathbf s}$. Hence the kernel of $f_{L,\mathbf s}$ is a T\nobreakdash-nilpotent ideal.  Further we observe that the $R_{L,\mathbf s}$-module $F_{L,\mathbf s}=R_{L,\mathbf s}\otimes_RF$ is flat (since the $R$-module $F$ is flat), and the $R'_{L,\mathbf s'}$-module $F'_{L,\mathbf s'}$ is isomorphic to $R'_{L,\mathbf s'}\otimes_{R_{L,\mathbf s}}F_{L,\mathbf s}$.

We have seen that the $R'_{L,\mathbf s'}$-module $F'_{L,\mathbf s'}$ is projective. By Proposition \ref{projective-reduction}, it follows that the $R_{L,\mathbf s}$-module $F_{L,\mathbf s}$ is projective.  Applying Theorem \ref{s-times-strongly-flat} again, we conclude that the $R$-module $F$ is $\mathbf S^\times$-strongly flat, hence finitely quite flat.
\end{proof}

\begin{cor} \label{locally-perfect-fqf}
All finitely generated, countably presented flat modules over locally perfect commutative rings are finitely quite flat.
\end{cor}

\begin{proof}
Let $R$ be a locally perfect commutative ring with the Jacobson radical $J$, and let $F$ be a finitely generated, countably presented flat $R$-module.
Then $F/JF$ is a finitely generated, countably presented module over a von Neumann regular ring $R/J$.
Following the first (or the second) proof of Theorem~\ref{vNRs-are-CFQ}, all finitely generated, countably presented modules over von Neumann regular rings are finitely quite flat.
Thus the $R/J$-module $F/JF$ is finitely quite flat.  Since the ideal $J$ is T-nilpotent, Theorem \ref{finitely-quite-flat-reduction} is applicable, telling that the $R$-module $F$ is finitely quite flat.
\end{proof}

\end{document}